\documentclass[12pt]{amsart}
\usepackage{amscd,amsmath,amsthm,amssymb}
\usepackage{pstcol,pst-plot,pst-3d}
\usepackage{lmodern,pst-node}
\usepackage{multicol}
 \usepackage[dvips]{graphicx}
 \usepackage{epstopdf}
 \usepackage{graphicx}
\usepackage{amsfonts,amssymb,amscd,amsmath,enumerate,verbatim}
\psset{unit=0.7cm,linewidth=0.8pt,arrowsize=2.5pt 4}

\newpsstyle{fatline}{linewidth=1.5pt}
\newpsstyle{fyp}{fillstyle=solid,fillcolor=verylight}
\definecolor{verylight}{gray}{0.97}
\definecolor{light}{gray}{0.9}
\definecolor{medium}{gray}{0.85}

\def\frk{\mathfrak}               

\def\pp{{\frk p}}

\def\mm{{\frk m}}

\def\Phi{{\frk N}}

\def\opn#1#2{\def#1{\operatorname{#2}}} 
\opn\chara{char} \opn\length{\ell} \opn\pd{pd} \opn\rk{rk}
\opn\projdim{proj\,dim} \opn\injdim{inj\,dim} \opn\rank{rank}
\opn\depth{depth} \opn\grade{grade} \opn\height{height}
\opn\size{size}
\opn\embdim{emb\,dim} \opn\codim{codim}

\opn\Tr{Tr} \opn\bigrank{big\,rank}
\opn\superheight{superheight}\opn\lcm{lcm}
\opn\trdeg{tr\,deg}
\opn\reg{reg} \opn\lreg{lreg} \opn\ini{in} \opn\lpd{lpd}
\opn\size{size}\opn{\mult}{mult}
\opn{\Cl}{Cl}\opn{\trdeg}{trdeg}
\opn\div{div} \opn\Div{Div} \opn\cl{cl} \opn\Cl{Cl}
\opn\Spec{Spec} \opn\Supp{Supp} \opn\supp{supp} \opn\Sing{Sing}
\opn\Ass{Ass} \opn\Min{Min} \opn\cl{cl}
\opn\Ann{Ann} \opn\Rad{Rad} \opn\Soc{Soc}
\opn\Syz{Syz} \opn\Im{Im} \opn\Ker{Ker} \opn\Coker{Coker}
\opn\Am{Am} \opn\Hom{Hom} \opn\Tor{Tor} \opn\Ext{Ext}
\opn\End{End} \opn\Aut{Aut} \opn\id{id} \opn\ini{in}

\opn\nat{nat}
\opn\pff{pf}
\opn\Pf{Pf} \opn\GL{GL} \opn\SL{SL} \opn\mod{mod} \opn\ord{ord}
\opn\Gin{Gin}
\opn\Hilb{Hilb}\opn\adeg{adeg}\opn\std{std}\opn\ip{infpt}
\opn\Pol{Pol}
\opn\sat{sat}
\opn\Var{Var}
\opn\Gen{Gen}
\opn\lex{lex}
\opn\div{div}

\opn\aff{aff} \opn\con{conv} \opn\relint{relint} \opn\st{st}
\opn\lk{lk} \opn\cn{cn} \opn\core{core} \opn\vol{vol}
\opn\link{link} \opn\star{star}
\opn\gr{gr}

\def\pot#1#2{#1[\kern-0.28ex[#2]\kern-0.28ex]}

\opn\dirlim{\underrightarrow{\lim}}
\opn\inivlim{\underleftarrow{\lim}}
\let\union=\cup

\let\to=\rightarrow

\def\Implies{\ifmmode\Longrightarrow \else
        \unskip${}\Longrightarrow{}$\ignorespaces\fi}
\def\implies{\ifmmode\Rightarrow \else
        \unskip${}\Rightarrow{}$\ignorespaces\fi}
\def\iff{\ifmmode\Longleftrightarrow \else
        \unskip${}\Longleftrightarrow{}$\ignorespaces\fi}

\let\:=\colon
\newtheorem{Theorem}{Theorem}[section]

\newtheorem{Corollary}[Theorem]{Corollary}
\newtheorem{Proposition}[Theorem]{Proposition}

\let\epsilon\varepsilon
\let\phi=\varphi
\let\kappa=\varkappa
\def\qed{\ifhmode\textqed\fi
      \ifmmode\ifinner\quad\qedsymbol\else\dispqed\fi\fi}
\def\textqed{\unskip\nobreak\penalty50
       \hskip2em\hbox{}\nobreak\hfil\qedsymbol
       \parfillskip=0pt \finalhyphendemerits=0}
\def\dispqed{\rlap{\qquad\qedsymbol}}

\opn\dis{dis}
\def\pnt{{\raise0.5mm\hbox{\large\bf.}}}

\opn\Lex{Lex}
\opn\int{int}

\newcommand{\inD}[1][\relax]{\def\argone{#1}\def\temprelax{\relax}
  \ifx\argone\temprelax\right.\else\,\middle|#1\right.{}\fi}

\begin{document}
\title {Alexander duality for monomial ideals associated with isotone maps between posets}
\author {J\"urgen Herzog, Ayesha Asloob Qureshi and Akihiro Shikama}
\thanks{This paper was partially written during the visit of the  second and third author at Universit\"at Duisburg-Essen, Campus Essen. The second author was supported by JSPS Postdoctoral Fellowship Program for Foreign Researchers
}


\address{J\"urgen Herzog, Fachbereich Mathematik, Universit\"at Duisburg-Essen, Campus Essen, 45117
Essen, Germany} \email{juergen.herzog@uni-essen.de}

\address{Ayesha Asloob Qureshi, Department of Pure and Applied Mathematics, Graduate School of Information Science and Technology,
Osaka University, Toyonaka, Osaka 560-0043, Japan}
\email{ayesqi@gmail.com}
\address{Akihiro Shikama, Department of Pure and Applied Mathematics, Graduate School of Information Science and Technology,
Osaka University, Toyonaka, Osaka 560-0043, Japan}
\email{a-shikama@cr.math.sci.osaka-u.ac.jp}

\begin{abstract}
For a pair $(P,Q)$ of finite posets the generators of the ideal $L(P,Q)$ correspond bijectively to the isotone maps  from $P$ to $Q$. In this note we determine all pairs $(P,Q)$ for which the Alexander dual of $L(P,Q)$ coincides with $L(Q,P)$, up to a switch of the indices.
\end{abstract}
\maketitle

\section*{Introduction}
In \cite{HH}, Hibi and the first author introduced a class of monomial  ideals which nowadays are called Hibi ideals. Given a finite poset $P$,  the generators of Hibi ideals are squarefree monomials which correspond bijectively  to the poset ideals of $P$. Later this class of ideals was generalized by Ene, Mohammadi and the first author in \cite{EHM} by considering  squarefree monomial ideals whose generators correspond  to the  chains of poset ideals of given length in $P$. The ideals generated by such monomials are called generalized Hibi ideals. In that paper, the Alexander dual of a generalized Hibi ideal is determined and  is identified as a multichain ideal. The concept of generalized Hibi ideals and multichain ideals has been further generalized in \cite{FGH}  by  Fl{\o}ystad, Greve and the first author. To describe this class of ideals, let $P$ and $Q$ be finite posets. A map $\phi: P \rightarrow Q$ is called {\em isotone} if  it is order preserving. In other words, $\phi: P \rightarrow Q$ is isotone if and only if $\phi(p_1) \leq \phi(p_2) $ for all $p_1, p_2 \in P$ with $p_1 \leq p_2$. The set of isotone maps $P \rightarrow Q$ is denoted by $\Hom(P,Q)$. Now let $K$ be a field and $S$ be the polynomial ring over $K$ in the  indeterminates $x_{pq}$ with $p \in P$ and $q \in Q$. As in \cite{FGH}, we denote by $L(P,Q)$ the ideal generated by the monomials $u_{\phi}=\prod_{p \in P} x_{p \phi(p)}$ where $\phi \in \Hom(P,Q)$. Let $[n]$ be the totally ordered poset  with $1 < 2 < \cdots < n$. It is easily seen that a generalized Hibi ideal on $P$  is of the form $L(P,[n])$ while a multichain ideal on $Q$ is of the form $L([n],Q)$. In \cite{FGH}, the ideals $L([n],Q)$ and $L(P,[n])$ are called letterplace and co-letterplace ideals, respectively. The classical Hibi ideals can be identified with $L(P,[2])$.

According to Theorem 1.1 in \cite{EHM},  the Alexander dual $L (P,[n])^{\vee} $ of $L(P,[n])$ is equal to the ideal $L([n],P)^{\tau}$. Here, for any $P$ and $Q$, $L(Q,P)^{\tau}$ is obtained from $L(Q,P)$ by switching the indices. In the other words, $$L(Q,P)^{\tau}= (\prod_{q \in Q} x_{\psi(q)q} \:\;  \psi \in \Hom(Q,P) ).$$
An alternative proof of this fact is given \cite[Proposition 1.2]{FGH}. Since $(I^\vee)^\vee =I$ for any squarefree monomial ideal, one also has $L ([n], P)^{\vee}=L(P,[n])^\tau$.

One would expect that  $L(P,Q)^{\vee} = L(Q,P)^\tau$ in general. Unfortunately, this is not always the case as can be shown by simple examples. In this paper we determine all pairs $(P,Q)$ of posets for which this duality holds. For this classification, we use \cite[Lemma 1.1]{FGH} which says that  any isotone map $\phi\: P\to P$ of a finite poset $P$  has a fixpoint, given that $P$ has a unique minimal or maximal element.

\section{Alexander duality for $L(P,Q)$}

All posets considered in this paper are assumed to be finite. Recall that the direct sum of two posets $P$ and $Q$ on disjoint sets is the  poset $P+Q$ on $P \cup Q$ such that $x \leq y$ in $P+Q$ if either $x,y \in P$ and $x \leq y$ in $P$ or $x,y \in Q$ and $x \leq y$ in $Q$. A poset is called {\em connected} if it is  not a direct sum of two posets. Alternatively, $P$ is connected if for any $a,b \in P$, there exists a finite sequence $a=a_1, a_2, \ldots, a_n=b$ in $P$ such that $a_i$ and $a_{i+1}$ are comparable in $P$ for $i=1, \ldots, n-1$.

Let $\Min(L(P,Q))$ denotes the set of minimal prime ideals of $L(P,Q)$. By using \cite[Corollary 1.5.5]{HHBook} it follows immediately that
$L(P,Q)^\vee=L(Q,P)^\tau$ if and only if
\begin{eqnarray}
\label{min}
\Min(L(P,Q))=\{\pp_\psi\: \psi\in \Hom(Q,P)\},
\end{eqnarray}
where
\begin{eqnarray}
\label{psi}
\pp_\psi=(x_{\pi(q)q}\:\; q\in Q).
\end{eqnarray}

In \cite[Proposition 1.5]{FGH} the following result is shown

\begin{Proposition}
\label{fgh}
Let $P$ and $Q$ be posets and assume that $P$ has a unique maximal or minimal element. Then for any  $\pp\in \Min(L(P,Q))$ with $\height \pp\leq |Q|$, there exists  $\psi\in \Hom(Q,P)$ such that $\pp=\pp_\psi$, and $\pp_\psi\neq \pp_{\psi'}$ for $\psi,\psi'\in \Hom(Q,P)$ with $\psi\neq \psi'$.
\end{Proposition}


\medskip
As an immediate consequence of Proposition~\ref{fgh} one obtains

\begin{Corollary}
\label{height}
Let $P$ and $Q$ be posets and assume that $P$ has a unique maximal or minimal element. Then
\begin{enumerate}
\item[{\em (a)}] $\height L(P,Q)=|Q|$;

\item[{\em (b)}]  $L(P,Q)^\vee=L(Q,P)^\tau$  if and only if $\height \pp=|Q|$ for all $\pp\in \Min(L(P,Q))$.
\end{enumerate}
\end{Corollary}

We first show

\begin{Proposition}
\label{complete}
Let $P$ and $Q$ be posets such that $L(P,Q)^\vee=L(Q,P)^\tau$. Then $P$ or $Q$ is connected.
\end{Proposition}

\begin{proof}
Suppose that $P$ and $Q$ are both disconnected. Then there exists posets $P_1, P_2$ and $Q_1, Q_2$ such that $P=P_1 + P_2$ and $Q= Q_1 +Q_2$ with  posets $P_1, P_2$ and $Q_1,Q_2$. Since $L(P,Q)^\vee=L(Q,P)^\tau$ it follows that $\Min (L (P,Q))=\{\pp_\psi\:\; \psi\in \Hom(Q,P)\}$. Let $p_1\in P_1$ and $p_2\in P_2$. Then the map
\[
\psi(q)= \left\{ \begin{array}{ll}
          p_1, &  \text{if  $q \in Q_1$}, \\
        p_2, &\text{if  $q \in Q_2$}
                  \end{array} \right.
\]
is isotone, and hence
\[
\pp_\psi=(\{x_{p_1q}\:\, q\in Q_1\}\union \{x_{p_2q}\:\, q\in Q_2\})
\]
is a minimal prime ideal of $L(P,Q)$.

On the other hand,  let $q_1\in Q_1$  and $q_2\in Q_2$ and let
\[
\phi(p)= \left\{ \begin{array}{ll}
          p_2, &  \text{if  $p \in Q_1$}, \\
        p_1, &\text{if  $p \in Q_2.$}
                  \end{array} \right.
\]
Then $\phi\:\; P\to Q$ is isotone and hence $u_\phi=\prod_{p\in P_1}x_{pq_2}\prod_{p\in P_2}x_{pq_1}$ belongs to $L(P,Q)$, while $u_\phi\not\in \pp_\psi$, a contradiction.
\end{proof}

In further discussion we may assume that $P$ or $Q$ is connected.  In the next statement  we will assume that $P$ is connected.

\medskip
We call $P$ a {\em rooted} poset if for any two incomparable elements $p_1, p_2 \in P$ there is no element $p \in P$ such that $p > p_1, p_2$. Similarly we call $P$ a {\em co-rooted} poset if for any two incomparable elements $p_1, p_2 \in P$ there is no element $p \in P$ such that $p < p_1, p_2$.  Note that a poset which  is rooted and co-rooted  is a finite direct sum of totally ordered posets. Also, observe that if $P$ is connected and rooted then $P$ has a unique minimal element. Indeed, if $P$ has two distinct minimal element, say $a,b \in P$, then by using the definition of connected poset, we obtain a sequence $a=a_1, a_2,  \ldots, a_n=b$ in $P$ such that $a_i$ and $a_{i+1}$ are comparbable, for all $i= 1, \ldots, n-1$. This sequence is not a chain because $a$ and $b$ are incomparable. Thus, there exist three distinct elements $a_{i-1} <  a_{i} > a_{i+1}$ for some $i=2 \ldots, n-1$, which contradicts the definition of rooted poset. Similarly, if $P$ is connected and co-rooted then $P$ has a unique maximal element.

\begin{Theorem}
\label{dual}
Let $P$ and $Q$ be finite posets, and assume that $P$ is connected but not a chain.
\begin{enumerate}
\item[{\em(a)}] If P is rooted, then $L(P,Q)^{\vee} = L(Q,P)^{\tau}$ if and only if $Q$ is rooted.
\item[{\em(b)}] If P is co-rooted, then $L(P,Q)^{\vee} = L(Q,P)^{\tau}$ if and only if $Q$ is co-rooted.
\item[{\em(c)}] If $P$ is neither rooted nor co-rooted, then $L(P,Q)^{\vee} = L(Q,P)^{\tau}$ if and only if $Q$ is a direct sum of chains.
\end{enumerate}
\end{Theorem}

\begin{proof}
(a) Assume that $L(P,Q)^{\vee} = L(Q,P)^{\tau}$ and that $Q$ is not rooted. Then there exists $q_1,q_2,q_3 \in Q$ such that $q_1$ and $q_2$ are incomparable and $q_1,q_2 <q_3$. Let $p_1, p_2 \in P$ be a pair of incomparable elements. Since $P$ is rooted and not a chain, there exists $p_3 \in P$ such that $p_3 < p_1, p_2$. We claim that
\[
\pp=(\{x_{p_1q}\; : \; q \geq q_1\} \cup \{x_{p_2q}\; :\;q \geq q_2\} \cup \{x_{p_3q}\;: \; q\ngeq q_1  \text{ and } q \ngeq q_2 \})
\]
is a minimal prime ideal of $L(P,Q)$ with $\height \pp > |Q|$. This will provide a contradiction to Corollary~\ref{height}(b).

To prove our claim, we first show that $L(P,Q) \subset \pp$. Assume that there exists $\phi \in \Hom(P,Q)$ such that  $u_\phi \notin \pp$. Then $\phi(p_1) \ngeq q_1$ and $\phi(p_2) \ngeq q_2$, and moreover, $\phi(p_3) \geq q_1$ or $\phi(p_3) \geq q_2$. We may assume that $\phi(p_3) \geq  q_1$. Then $ q_1 \leq \phi(p_3) \leq \phi(p_1)$ contradicting the fact that $\phi(p_1) \ngeq q_1$. Hence, $L(P,Q) \subset \pp$.

Now we show that $\pp$ is a minimal prime ideal of $L(P,Q)$.  Due to Corollary~\ref{height}, for all $q \in Q$, there exists $p \in P$ such that $x_{pq} \in \pp$.  This implies that we can not skip the variable $x_{pq}$ from generators of $ \pp$ if $q$ appears only once as the second index. Assume now that $q \in Q$ appears twice as a second index among the generators of $\pp$. Then $q > q_1, q_2$ and $x_{p_1q}, x_{p_2q} \in \pp$. Now we show that we can not skip any of $x_{p_1q}$ or $x_{p_2q}$ from the set of generators of $\pp$.


Indeed, let $\psi: P \rightarrow Q$ given by

\[
\psi(p)= \left\{ \begin{array}{ll}
          q, &  \text{if  $p \geq p_1$}, \\
        q_1, &\text{otherwise}.
                  \end{array} \right.
\]

Note that $\psi$ is an isotone map. In fact, let $p , p' \in P$ with $p \geq p'$. We have to show that $\psi(p) \geq \psi (p')$. This is obvious  if  $p,p' \geq p_1$ or $p, p' \ngeq p_1$. The only case which remains is that $p \geq p_1, p' \ngeq p_1$. But in this case we have $\phi(p) = q > q_1 = \phi(p')$.

Since $\psi $ is an isotone map, it follows that $u_{\psi} \in \pp$. Since $x_{p_1q}$ is the only generator of $\pp$ which divides $u_{\psi}$, this generator of $\pp$ can not be skipped. Similarly, one can show that $x_{p_2q}$ can not be skipped as a generator of $\pp$. It shows that $\pp$ is indeed a minimal prime ideal of $L(P,Q)$.

Conversely,  suppose that $Q $ is a rooted poset and $L(P,Q)^{\vee} \neq L(Q,P)^{\tau}$. Then by using Corollary~\ref{height} (b), we see that there exists a minimal prime ideal $\pp$ of $L(P,Q)$ with $\height \pp >|Q|$. This implies that there exists an element $q \in Q$ such that $x_{p_1q}, x_{p_2q} \in \pp$ for some $p_1, p_2 \in P$ with $p_1 \neq p_2$. Since $\pp$ is a minimal prime ideal,  neither $x_{p_1q}$ nor $x_{p_2q}$ can be skipped from the set of generators of $\pp$. It implies that there exist $\phi_1, \phi_2 \in \Hom(P,Q)$ such that $x_{p_1q}$ is the only generator of $\pp$ which divides $u_{\phi_1}$ and $x_{p_2q}$ is the only generator of $\pp$ which divides $u_{\phi_2}$.

Suppose first that $p_1$ and $p_2$ are comparable. We may assume that $ p_2> p_1$. Then $\phi_1(p_2) > q= \phi_1 (p_1)$, otherwise $u_{\phi_1}$ is divisible by both $x_{p_1q}$ and $x_{p_2q}$. Similarly, $\phi_2(p_1) < q= \phi_2 (p_2)$.

Let $\psi : P \rightarrow Q$ given by

 \[
\psi(p)= \left\{ \begin{array}{ll}
          \phi_1(p), &  \text{if  $p \geq p_2$}, \\
         \phi_2 (p), &\text{otherwise}.
                  \end{array} \right.
\]

We claim that $\psi $ is an isotone map. To see this it suffices to show that $\psi(p) \geq \psi(p')$ for $p > p'$ and $p \geq p_2$, $p' \ngeq p_2$. Suppose that  $p' < p_2$ then $\psi(p) = \phi_1 (p) > q > \phi_2(p) = \psi(p') $. Suppose that $p' \not< p_2$, then $p'$ and $p_2$ are incomparable. This case is not possible since $p > p_2$ and $p > p'$ and since $P$ is rooted.

Following the construction of $\psi$, we see that  $u_\psi \notin \pp$. This contradicts the fact that $L(P,Q) \subset \pp$.

\medskip

Finally assume that $p_1$ and $p_2$ are incomparable.  Since $P$ is rooted, there exists a unique maximal element $p_3 \in P$ such that $p_3 < p_1, p_2$.  Therefore, $\phi_1 (p_3), \phi_2(p_3) \leq q$. Since $Q$ is rooted, it follows that $\phi_1 (p_3), \phi_2(p_3)$ are comparable. We may assume that $\phi_1 (p_3) \leq \phi_2(p_3)$.

There exists a unique element  $p_4$ with the property $p_3 < p_4 \leq p_1$, because $P$ is rooted.  Let $\psi : P \rightarrow Q$ given by

 \[
\psi(p)= \left\{ \begin{array}{ll}
          \phi_2(p), &  \text{if  $p_4 \leq p$}, \\
         \phi_1 (p), &\text{otherwise}.
                  \end{array} \right.
\]

We claim that $\psi$ is an isotone map. To prove this it suffices to show that $\psi (p) > \psi(p')$ for $p > p'$ with $p \geq p_4$, $p' \ngeq p_4$. If $p' < p_4$ then note that $p' \leq p_3 < p_4$ because $P$ is rooted. Then $\psi(p)=\phi_2(p) \geq \phi_2(p_4)  \geq \phi_2 (p_3) \geq \phi_1(p_3) \geq \phi_1(p')= \psi(p')$ . If $p'$ and $p_4$ are incomparable then $p$ and $p'$ are  also incomparable because $P$ is rooted. It shows that $\psi$ is an isotone map and $u_\psi \notin \pp$, a contradiction.

Statement (b) is proved in the same way.

(c) Let $L(P,Q)^{\vee} = L(Q,P)^{\tau}$ and assume that $Q$ is not a direct sum of chains. Then there exists $ q_1, q_2, q_3 \in Q$ such that $q_1$ and $q_2$ are incomparable and either $q_3 < q_1, q_2$ or $q_1, q_2 < q_3$. Assume that $q_1, q_2 < q_3$. Since $P$ is neither rooted nor co-rooted we have $p_1, p_2, p_3$ such that $p_1$ and $p_2$ are incomparable and $p_3 < p_1, p_2$.Then by following the proof of (a) we obtain a minimal prime ideal of $L(P,Q)$ of height greater than $|Q|$, which is not possible. Similarly, one can show that it is not possible to have  $ q_1, q_2, q_3 \in Q$ such that $q_1$ and $q_2$ are incomparable and $q_3 < q_1, q_2$. It follows that $Q$ is a direct sum of chains.

Conversely, assume that $Q$ is the direct sum of the chains $Q_1, Q_2, \ldots, Q_n$. Then $L(P,Q)^\vee = (L(P,Q_1) + \cdots + L(P,Q_n))^\vee = L(P,Q_1)^\vee  \cdots  L(P,Q_n)^\vee$. By \cite[Proposition 1.2]{FGH},  $L(P,Q_i) = L(Q_i, P)^\tau$. Therefore,
\[
L(P,Q)^\vee= \prod_{i=i}^n L(Q_i, P)^\tau = (\prod_{i=i}^n L(Q_i, P))^\tau = L(Q,P)^\tau
\]
\end{proof}

As the  final conclusion we obtain

\begin{Corollary}
\label{final} The following conditions are equivalent:
\begin{enumerate}
\item[{\em (a)}] $L(P,Q)^\vee=L(Q,P)^\tau$.
\item[{\em (b)}]  $P$ is connected or $Q$ is connected, and  one of the following conditions is satisfied:
\begin{enumerate}
\item[{\em (i)}] $P$ and $Q$ are rooted;
\item[{\em (ii)}] $P$ and $Q$ are co-rooted;
\item[{\em (iii)}] $P$ is connected  and $Q$ is a sum of chains;
\item[{\em (iv)}]$Q$ is connected and $P$ is a sum of chains;
\item[{\em (v)}] $P$ is a chain or $Q$ is a chain.
\end{enumerate}
\end{enumerate}
\end{Corollary}

\begin{proof}
The result follows \cite[Theorem 1.1]{EHM}, Proposition~\ref{complete} and Theorem~\ref{dual} observing that
\begin{eqnarray}
\label{selfie}
L(P,Q)^\vee =L(Q,P)^\tau \quad \iff \quad L(Q,P)^\vee =L(P,Q)^\tau.
\end{eqnarray}
The statement (\ref{selfie}) is a consequence of  the fact that Alexander duality as well as the operator $\tau$ are involutary and commute with each other. Thus if $L(P,Q)^\vee =L(Q,P)^\tau$, then
\[
(L(Q,P)^\vee)^\tau=(L(Q,P)^\vee)^\tau=(L(P,Q)^\tau)^\tau=L(P,Q),
\]
which implies that $L(Q,P)^\vee=L(P,Q)^\tau$. This show ``\implies". The other direction follows by symmetry.
\end{proof}

\end{document}